\documentclass[10pt,a4paper]{article}
\usepackage{amsfonts}
\usepackage{amssymb,amsthm, amsmath,graphicx}
\usepackage{bbm}

\usepackage{url}

\newtheorem{theorem}{Theorem}

\newtheorem{proposition}[theorem]{Proposition}
\newtheorem{corollary}[theorem]{Corollary}
\newtheorem{lemma}[theorem]{Lemma}

\theoremstyle{remark}

\newtheorem{example}[theorem]{Example}

\newtheorem{remark}[theorem]{Remark}

\def\CaF{\mathbf{F}}

\def\CaP{\mathbf{P}}

\def\FraC{\mathcal{C}}
\def\FraF{\mathcal{F}}
\def\FraI{\mathcal{I}}
\def\FraS{\mathcal{S}}
\def\FraP{\mathcal{P}}

\def\N{\mathbb{N}}
\def\R{\mathbb{R}}
\def\Z{\mathbb{Z}}
\def\Q{\mathbb{Q}}
\def\d{\mathrm{d}}

\def\int{\mathrm{int} }

\title{Affine convex body semigroups}

\author{J. I. Garc\'{\i}a-Garc\'{\i}a\footnote{Departamento de Matem\'aticas, Universidad de C\'adiz,
E-11510 Puerto Real (C\'{a}diz, Spain). E-mail: ignacio.garcia@uca.es. Partially supported by MTM2007-62346 and Junta de Andaluc\'{\i}a group FQM-366. }\\
M.A. Moreno-Fr\'{\i}as\footnote{Departamento de Matem\'aticas, Universidad de C\'adiz,
E-11510 Puerto Real (C\'{a}diz, Spain). E-mail: mariangeles.moreno@uca.es. Partially supported by MTM2008-06201-C02-02 and Junta de Andaluc\'{\i}a group FQM-298.}\\
A. S\'{a}nchez-R.-Navarro \footnote{Departamento Lenguajes y Sistemas Inform\'{a}ticos, Universidad de C\'adiz,
E-11405 Jerez de la Frontera (C\'{a}diz, Spain). E-mail: alfredo.sanchez@uca.es. Partially supported by Junta de Andaluc\'{\i}a group FQM-366.}\\
A. Vigneron-Tenorio\footnote{Departamento de Matem\'aticas, Universidad de C\'adiz,
E-11405 Jerez de la Frontera (C\'{a}diz, Spain). E-mail: alberto.vigneron@uca.es. Partially supported by MTM2007-64704 and Junta de Andaluc\'{\i}a group FQM-366.}\\
}

\date{}

\begin{document}

\maketitle

\begin{abstract}
In this paper we present a new kind of semigroups called convex body semigroups which are generated by convex bodies of $\R^k$.
They generalize to arbitrary dimension the concept of proportionally modular numerical semigroup of \cite{Rosales_modular}.
Several properties of these semigroups are proven. Affine convex body semigroups obtained from circles and polygons of $\R^2$ are characterized. The algorithms for computing minimal system of generators of these semigroups are given. We provide  the implementation of some of them.

\smallskip
{\small \emph{Keywords:} Affine semigroup, circle semigroup, convex body monoid, convex body semigroup, polygonal semigroup.}

\smallskip
{\small \emph{MSC-class:} 20M14 (Primary),  20M05 (Secondary).}
\end{abstract}

\section*{Introduction}
Let $F$ be a subset of $\R^k$, $\CaF=\bigcup_{i=0}^{\infty} F_i\cap \R^k_{\geq}$ and   $\FraF=\bigcup_{i=0}^{\infty} F_i\cap \N^k,$
where $F_i=\{iX| X\in F\}$ with $i\in \N$. A convex body of $\R^n$ is a  compact convex subset with non-empty interior.
If $F$ is a convex body, then the set $\CaF$ is a monoid and $\FraF$ is a semigroup  (see Proposition \ref{pr1}). Given a convex body $F$, we call convex body monoid (respectively semigroup) generated by $F$ to the above monoid (respectively semigroup) $\CaF$ (respectively $\FraF$). In this work we consider the usual topology of $\R^k$.

In general these semigroups are not finitely generated. If $\FraF$ is a finitely generated semigroup we say that $\FraF$ is an affine convex body semigroup. Given a convex polygon or a circle in $\R^2,$ we study the necessary and sufficient conditions for $\FraF$ to be finitely generated. These conditions are related to the slopes of the extremal rays of the minimal cone which includes to $\FraF.$ We give effective methods to obtain their minimal system of generators.

In \cite{Rosales_modular}, the authors present the numerical monoids and semigroups generated by intervals ($F=[\alpha,\beta]\subseteq \R_{\geq}$ with $\alpha<\beta$) called proportionally modular numerical semigroups. They prove proportionally modular numerical semigroups are characterized by a modular Diophantine inequality (see \cite[Theorem 8]{Rosales_modular}). We generalize this modular Diophantine inequality for the convex body monoids and semigroups (see Corollaries  \ref{necesaria_y_suficiente_ecuacion} and  \ref{ecuacion_circulos}).

The minimal system of generators of a proportionally modular numerical semigroup can be obtained by constructing a B\'{e}zout sequence connecting two reduced fractions (see \cite{Rosales_bezout} and \cite{Rosales_modular}). In Lemma \ref{generadores_rayos} it can be found an alternative method to compute this minimal system of generators.

Besides, Lemma \ref{pertenencia_circulo} shows an easy algorithm to check if an element belongs to a circle semigroup, and Corollary \ref{bound} provides a bound for the minimal generators of these semigroups. The implementation of the algorithm to compute the minimal system of generators of a circle semigroup is available at the url \cite{programa}.

The contents of this work are organized as follows.
In Section \ref{s2} we give some concepts and results used during this work.
We also characterize convex body semigroups in terms of Diophantine inequalities.
In Section \ref{s3} some algebraic and geometrical constructions are given.
Section \ref{s4} and \ref{s5} are devoted to characterize the affine semigroups generated by a polygon (polygonal semigroup) or a circle (circle semigroup). The algorithms to compute their minimal systems of generators are showed.
For theses cases in Section \ref{s6} we compute a bound for the minimal generators of the affine semigroup.

\section{Convex semigroups}\label{s2}
Given $\{a_1,\ldots, a_r\}\subseteq \N^k$, we denote by $S=\langle a_1,\ldots, a_r\rangle$ the subsemigroup of $\N^k$ generated by $\{a_1,\ldots, a_r\}$, that is, $\langle a_1,\ldots, a_r\rangle=\{\lambda_1a_1+\cdots+\lambda_ra_r|\, \lambda_1,\ldots,\lambda_r\in \N\}$.
If no proper subset of $\{ a_1,\ldots, a_r\}$ generates $S$, then this set is called the minimal system of generators of $S$.
Every affine semigroup admits a unique minimal generating system (see \cite{Rosales3}).

Define the cone generated by $A\subseteq R^k_{\geq}$ as the set
$$
L_{\Q_{\geq}}(A)=\left\{\sum_{i=1}^p q_ia_i| p\in\N, q_i\in \Q_{\geq}, a_i\in A \right\}.
$$
A ray is a line containing the zero element, $O,$ of $\R^k$. A ray is defined by only one point not equal to $O$.
Given $A\subseteq \R^2_{\geq},$ denote by $\tau_1$ and $\tau_2$ to the extremal rays of $L_{\Q_{\geq}}(A)$ (assume the slope of $\tau_1$ is greater than the slope of $\tau_2$), and by $ \int(A)=A \cap (L_{\Q_{\geq}}(A)\setminus \{\tau_1,\tau_2\})$. We called interior of $A$ to the set $\int(A)$.

Let $F$  be a convex body of $\R^k$ and let $$\CaF=\{X\in \R^k _{\geq}| \textrm{ there exists } i\in \N \textrm{ such that } \frac{X}{i}\in F\}\cup \{0\}=\bigcup_{i=0}^{\infty} F_i,$$
where $F_i=\{iX| X\in F\}$ with $i\in \N$.

\begin{proposition}\label{pr1}
$\CaF$ is a submonoid of $\R ^k.$
\end{proposition}

\begin{proof}
Let $P,Q\in\CaF$. There exist $i,j\in \N$ and $P',Q'\in F$ such that $P=iP'$ and $Q=jQ'$. Then
$$
P+Q=iP'+jQ'=(i+j)\left(\frac{i}{i+j}P'+(1-\frac{i}{i+j})Q'\right).
$$
Using the convexity of $F$ we obtain $\frac{i}{i+j}P'+(1-\frac{i}{i+j})Q'\in F$ and so $P+Q\in \CaF$.
\end{proof}

We call convex body monoid of $\R^k$ to every submonoid $\CaF$ of $\R^k$ obtained as above from a convex body of $\R^k$.

Denote by $\d(P,Q)$ the Euclidean distance between two elements $P,Q\in\R^k$ and by $\d(P)$ the distance $\d(P,O)$. We see the convexity property is necessary to $\CaF$ be a monoid. If $F$ is the compact and not convex set $$\{X\in \R^2_{\geq}|3\leq \d(X)\leq 5\},$$  the elements $(4,0),(0,4)$ are in $\CaF$ but $(4,0)+(0,4)$ is not in $\CaF$.

Define a convex body semigroup as the intersection of a convex body monoid with $\N^k$. In general, these semigroups are not full affine semigroup, that is, they can not be expressed using linear Diophantine equations (see \cite{Rosales3}). To see this, consider a convex body $F$ of $\R^k$ fulfilling that it has at least an element $P$ satisfying that $P+e_1\in F$, where $e_1$ is the first element of the canonical basis of $\R^k$, and  $e_1\not\in F$. This implies the elements $P,P+e_1\in\CaF$ but $(P+e_1)-P=e_1\not \in\CaF$.

The following result is a generalization of Theorem 8 of \cite{Rosales_modular} and it provides an inequality which characterizes the elements of a convex body monoid of $\R^k$.

Observe that if a ray intersects with $F_1$ in only a point (respectively a segment), then  the intersection of the ray with any other $F_i$ with $i>1$ is also a point (respectively a segment). Denote by $\overline{PQ}$ the segment joining $P$ and $Q$.

\begin{proposition}\label{ecuacion}
Let $\tau$ be a non-negative slope ray. Then, for all $X\in\CaF\cap \tau$ there exist $a,b\in\R_{\geq}$ with $1<a<b$, such that
\begin{equation}\label{ecuacion_inicial}a\cdot\d(X) \mod b \le \d(X).\end{equation}
\end{proposition}

\begin{proof}
If $X\in \CaF \cap \tau$, then there exists $i\in\N$ such that $X\in F_i$. If $i=0$, then $X=0$ and there exist $a,b\in \R_{\geq}$ such that the inequality is clearly satisfied.

Assume that $X\in F_i$, with $i>0$.
Observe the intersection $\tau \cap F_i$, can be only a
point or a segment.
If $\tau \cap F_i=\{X\}$  then
there exists $P\in F$ such that $X=iP$ and $\d(X)=i\d(P).$
Taking now a number $a\in (1, \infty)$ we obtain $a<ai$ and
$a \d(X)\mod ai\d(P)=0\le \d(X).$
 If $\tau \cap F_i=\overline{PQ}$ (assume
$\d(P)<\d(Q)$), then $X\in i \overline{PQ}$ and  $\d(X)$ belongs to a submonoid of
$\R_{\geq}$ generated by $[\d(P),\d(Q)].$ By
\cite[Theorem 8]{Rosales_modular}, we conclude there exist
$a,b\in (1,\infty)$ with  $b>a$ such that $a\d(X)\mod b \le \d(X).$
\end{proof}

From the above proposition it can be deduced that $a$ and $b$ depend only of the vector $\overrightarrow{OX}$. This fact allows us to characterize the elements of a convex body semigroup from an inequality. Denote by $\tau$ the ray containing the point $X.$

\begin{corollary}\label{necesaria_y_suficiente_ecuacion}
An element $X\in \N^k$ belongs to $\int(\CaF)$ if and only if the following conditions are fulfilled:
\begin{enumerate}
\item $\tau \cap F$ is a segment $\overline{PQ}$ with $P,Q\in\int(\CaF)$.
\item $\displaystyle{\frac{\d(Q)}{\d(Q)-\d(P)}} \d(X) \mod  \displaystyle{\frac{\d(P)\d(Q)}{\d(Q)-\d(P)}} \le \d(X).$
\end{enumerate}
\end{corollary}

\begin{proof}
It is straightforward from Proposition \ref{ecuacion} and the proof of Theorem 8 in \cite{Rosales_modular}.
\end{proof}

\section{Tools}\label{s3}
Let $F$ be a convex body  of $\R_{\geq}^2$ and $\tau_1, \tau_2$  the extremal rays of $L_{\Q_{\geq}}(F)$ (assume the slope of $\tau_1$ is greater than the slope of $\tau_2$). Observe that $\CaF$ is contained in the cone $L_{\Q_{\geq}}(F).$ The subsemigroup $L_{\Q_{\geq}}(F)\cap\N^2$ is denoted by $\FraC$.

In general for every  semigroup equal to the set of non-negative integer solutions of a system of inequalities (for instance $\FraC$), its minimal system of generators can be determined by obtaining the minimal solutions of a system of Diophantine equations (see \cite{Evelyne2} and \cite{Nsol}).

\begin{lemma}\label{lema_rectangulo}
Let $\tau$ be a rational slope ray, $g, s\in \tau\cap \N^2$ and $\overrightarrow{u}\in\R^2$. Define $R_i$ the parallelogram determined by the elements
$g+(i-1)s$, $g+is$ and $g+(i-1)s + \overrightarrow{u}$ with $i\in \N$. If $R_1\subset\R^2_{\geq}$, then $R_{i}\cap \N^2=(R_{1}\cap \N^2)+(i-1)s$.
\end{lemma}

\begin{proof}
By construction $R_i=R_1+(i-1)s$ for every $i\in\N.$ Since $s\in\N^2$, then $R_{i}\cap \Z^2=(R_{1}\cap \Z^2)+(i-1)s.$ In case $R_1\subset \R^2_{\geq},$ we obtain that $R_{i}\cap \N^2=(R_{1}\cap \N^2)+(i-1)s.$
\end{proof}

\begin{lemma}\label{semigrupo_segmento_racional}
Let $P,Q\in \Q_{\geq}$ (respectively $P,Q\in\Q^2_{\geq}$). The semigroup $\FraI= \big(\bigcup _{i \in \N}i\overline{PQ}\big) \cap \N $ (respectively $\FraI= \big(\bigcup _{i \in \N}i\overline{PQ}\big) \cap \N ^ 2$) is finitely generated and there exists an algorithm to determine its minimal system of generators.
\end{lemma}

\begin{proof}
Assume that $\overline{PQ}\subset \R _{\geq}$. The elements $P'=(P,1)$ and $Q'=(Q,1)$ belong to $\Q^2_{\geq}$. Denote by $\FraC'$ the semigroup $L_{\mathbb{Q}_{\geq}}(\{P',Q'\})\cap \N^2$. The set $\FraC'$ is determined by the rational systems of inequalities given by the two rays containing the points $P'$ and $Q'$. Thus $\FraC'$ is finitely generated. The semigroup $\FraI$ is the projection onto the first coordinate of the elements of $\FraC'$ and therefore it is finitely generated.

Let consider now the case $\overline{PQ}\subset \R ^2_{\geq}.$ Define again $P'=(P,1)$ and $Q'=(Q,1)$ elements of $\Q^3$. Take now $\overrightarrow{u}$ a normal vector to the subspace $\langle \overrightarrow{OP'},\overrightarrow{OQ'}\rangle$ and two vectorial planes $\pi_1$ and $\pi_2$  generated by $\{\overrightarrow{OP'},\overrightarrow{u}\}$ and $\{\overrightarrow{OQ'},\overrightarrow{u}\}$ respectively.

Let $\FraC'$ be the semigroup finitely generated by the minimal solutions of the system of rational inequalities determined by the plane containing the points $\{O,P',Q'\},$ and the cone delimited by $\pi_1$ and $\pi_2.$ Since $\FraI$ is the projection onto the first and second coordinate of $\FraC'$, it is finitely generated.

In both cases the minimal system of generators of $\FraI$ is obtained by an effective way from the set given by the projection of a system of generators of $\FraC'.$ A minimal system of generators of $\FraC'$ can be computed from the solutions of a system of Diophantine inequalities.
\end{proof}

\begin{lemma}\label{generadores_rayos}
Let $\tau$ be a ray and $\overline{PQ}$ a segment over $\tau$ with $P,Q\in\R^2\setminus\Q^2$ (assume $\d(P)<\d(Q)$). Then the semigroup $\FraI=\left(\bigcup_{i\in\N}i\overline{PQ}\right)\cap\N^2$ is finitely generated and there exists an algorithm for computing its minimal system of generators.
\end{lemma}

\begin{proof}
If $\tau$ has negative or irrational slope then $\left( \bigcup _{i\in \N} i\overline{PQ} \right) \cap \N^2=\emptyset,$ and therefore the result is straightforward.

Assume the slope of $\tau$ is not negative and rational. Let $k$ be the smallest positive integer fulfilling that $kQ-(k+1)P\in \R ^2_{\geq}.$ By construction the integer $k$ exists  and it can be determined, then the ray with vertex $(k+1)P$ and determined by $\overrightarrow{PQ}$ is included in the monoid $\bigcup_{i\in\N}i\overline{PQ}$.

Let $T$ be the finite set $\overline{O((k+1)P)} \cap \N^2$ and let
$$d_1=\min \left(\bigcup_{i=1}^{k+1}\{d(H,iP)| H\in T\}\right)/(k+1),$$
$$d_2=\min \left(\bigcup_{i=1}^{k+1}\{d(H,iQ)| H\in T\}\right)/k.$$

Consider the segment $\overline{P ' Q'}$ with $P '=P-d_1\overrightarrow{u}$ and $Q'=Q+d_2\overrightarrow{u}$ where $\overrightarrow{u}$ is the unitary direction vector of $\tau.$

The segment $\overline{P ' Q'}$ verifies that $\overline{PQ} \subset \overline{P ' Q'}$ and that for every segment $\overline{P '' Q''}$ with $P''$ and $Q''$ rational points such that $\overline{PQ}\subset \overline{P '' Q''} \subset \overline{P ' Q'},$ we have that $\FraI=\left( \bigcup _{i\in \N} i\overline{P '' Q''}\right) \cap \N^2.$

Since $P''$ and $Q''$ are rational, by Lemma \ref{semigrupo_segmento_racional}, we conclude that $\FraI$ is finitely generated.

The minimal system of generators of $\FraI$ can be computed in an effective way following the steps of this proof:
\begin{itemize}
\item Compute the smallest $k\in\N$ such that $kQ-(k+1)P\in \R ^2_{\geq}.$
\item Compute the set $T$ and the values $d_1$ and $d_2$.
\item Compute the vector $\overrightarrow{u}$ and take the rational points $P''$ and $Q''$.
\item Apply Lemma \ref{semigrupo_segmento_racional}.
\end{itemize}
\end{proof}

In particular, the above result can be used to obtain a system of generators of a proportionally modular semigroup. This is an alternative method to the one presented in \cite{Rosales_modular}.

The following results are used to find system of generators of convex body semigroups.

\begin{lemma}\label{sg_S_prima2}
Let $\{g_1,\ldots ,g_p\}\subset \N^2$ be the minimal system of generators of a semigroup $\FraF$ and $\tau=g_1\Q$ an extremal ray of  $\FraF$.
Assume that $g_1$ generates $\N^2 \cap \tau$ and consider $\{s_1,\ldots ,s_t\}$ the minimal system of generators of a subsemigroup of $\N^2\cap \tau$. Let  $\FraF'$ be the semigroup generated by  $B=B_1\cup B_2$ with
\begin{equation*}
B_1=\Big\{s_1,\ldots, s_t,g_2,\ldots ,g_p\Big\},\, B_2=
\bigcup _{i=2}^p \{g_i+g_1, \ldots, g_i+(\lambda_t-1)g_1\}
,\end{equation*} where  $0<\lambda_1<\cdots < \lambda _t$ are the integers such that  $s_i=\lambda _i g_1.$ Then the semigroup $\FraF'$ verifies:
\begin{itemize}
\item $ \FraF'\cap \tau = \langle s_1,\ldots ,s_t \rangle.$
\item $\FraF'\setminus \tau = \FraF \setminus \tau.$
\end{itemize}
\end{lemma}

\begin{proof}
Clearly $\FraF'\cap \tau = \langle s_1,\ldots ,s_t \rangle$.

On the other hand, let $g\in \FraF \setminus \tau.$ There exist $\mu _1,\ldots ,\mu _p\in \N$ with $\sum _{i=2}^p \mu _i \neq 0,$ such that  $g=\sum _{i=1}^p \mu _i g_i.$ Without lost of generality we can assume that $\mu _2>1.$ There are three possibilities:
\begin{itemize}
\item If $\mu _1=0,$ then it is trivial that $g\in \FraF'\setminus \tau.$
\item If $ \lambda_t > \mu _1 >0,$ then $g=\underbrace{g_2+\mu _1g_1}_{\in B_2}+ (\mu _2-1)\underbrace{g_2}_{\in B_1} +\sum _{i=3}^p \mu _i \underbrace{g_i}_{\in B_1}.$
\item If $\mu _1\ge \lambda_t>0,$ then there exist $u,v\in \N$ such that $\mu _1=u\lambda _t + v,$ with $\lambda _t>v.$ Thus,  $g= u\underbrace{(\lambda _t g_1)}_{\in B_1}+\underbrace{g_2+v g_1}_{\in B_2}+ (\mu _2-1)\underbrace{g_2}_{\in B_1} +\sum _{i=3}^p \mu _i \underbrace{g_i}_{\in B_1}.$
\end{itemize}

In any of the above cases we obtain that $g\in \FraF'\setminus \tau$ and we can conclude that $\FraF'\setminus \tau = \FraF \setminus \tau$ (trivially $\FraF'\setminus \tau \subset  \FraF \setminus \tau$).
\end{proof}

\begin{lemma}\label{quitando elementos}
Let $\FraF\subset \N ^2$ be a finitely generated semigroup and $a\in \FraF.$ The set $\FraF\setminus \{a\}$ is a semigroup if and only if  $a$ is a minimal generator of  $\FraF.$ Besides if $B=\{a,f_2,\ldots ,f_t\}$ is the minimal system of generators of  $\FraF,$ then the semigroup  $\FraF\setminus \{a\}$ is generated by
$$\left\{ f_2,\ldots ,f_t, f_2+a,\ldots ,f_t+a, 2a, 3a \right\}.$$
\end{lemma}

\begin{proof}
Assume that  $\FraF\setminus \{a\}$ is a semigroup and that  $a$ is not a minimal generator of $\FraF.$ Then there exist  $a_1,a_2\in \FraF\setminus
\{a\}$ such that  $a=a_1+a_2,$ which contradicts the fact that  $\FraF\setminus \{a\}$ is a semigroup.

Conversely, assume that $a$ is a minimal generator of  $\FraF$ (remind the semigroup $\FraF$ has a unique system of generators).
To prove that $\FraF\setminus \{a\}$ is a semigroup it is only necessary to show that the addition is an operation on this set. Let $x,y \in \FraF\setminus \{a\}$, then  $x+y \in \FraF\setminus \{a\}$ (if not we have that  $x+y =a$, which is impossible because $a$ is a minimal generator of $\FraF$).

Let $B=\{a,f_2,\ldots ,f_t\}$ the minimal system of generators of  $\FraF$ (without lost of generality we assume that $a$ is the first element of  $B$). Trivially, $\{f_2,\ldots ,f_t, f_2+a,\ldots ,f_t+a, 2a, 3a\}\subset \FraF\setminus \{a\}.$
Let  $f\in \FraF\setminus \{a\}\subset \FraF,$ therefore $\exists \lambda ,\lambda _2,\ldots \lambda _t \in \N$ such that
$f=\lambda a +\sum _{i=2}^t\lambda _i f_i.$ If $\lambda \neq 1,$ there exist $\alpha,\beta \in \N$ verifying that $\lambda = 2\alpha + 3 \beta,$ thus
$$f=\lambda a +\sum _{i=2}^t\lambda _i f_i=\alpha (2a)+\beta (3a) +\sum _{i=2}^t\lambda _i f_i.$$ If $\lambda = 1,$ since $a\notin \FraF\setminus \{a\},$ there exists $\lambda _{i_0}\geq 1,$ such that
$$f=a +\sum _{i=2}^t\lambda _i f_i=(f_{i_0}+a)+(\lambda _{i_0}-1)f_{i_0}+\sum _{i=2,\, i\neq i_0}^t\lambda _i f_i.$$ In any case, $\left\{ f_2,\ldots ,f_t, f_2+a,\ldots ,f_t+a, 2a, 3a \right\}$ is a system of generators of $\FraF\setminus\{a\}.$
\end{proof}

\begin{corollary}\label{semigrupo menos un numero finito de puntos}
Let  $\FraF$ be a finitely generated semigroup and $A\subset \FraF$ be a finite subset. If $\FraF\setminus A$ is a semigroup, then  $\FraF\setminus A$ is a finitely generated semigroup. Furthermore, there exists an algorithm to compute a system of generators of $\FraF\setminus A$.
\end{corollary}

\begin{proof}
Assume that $A=\{a_1,\ldots ,a_n\}\subset \FraF$ and assume that $B$ is the minimal system of generators of $\FraF$.
Using the proof of Lemma  \ref{quitando elementos}, at least an element of $A$ must be an element of $B$.
Assume that  $a_1\in B$, then by Lemma \ref{quitando elementos} we obtain that $\FraF_1=\FraF\setminus \{a_1\}$ is a subsemigroup of $\N ^2.$ Denote by $B_1$ to the minimal system of generators of the semigroup $\FraF_1$ which is obtained from the system of generators of $\FraF_1$ constructed as in Lemma \ref{quitando elementos}. Using again the above reasoning with the sets $A_1=A\setminus \{a_1\},$ $\FraF_1$ and $B_1,$
we obtain a new semigroup $\FraF_2=\FraF_1\setminus \{a_i\},$ where $a_i\in A_1\cap B_1$ with $i\in \{2,\ldots,n\}$.
Since $A$ is finite, this method stops after a finite number of steps and we obtain a finite system of generators $B_n$ of the semigroup $\FraF_n=\FraF\setminus A.$
\end{proof}

\section{Convex polygonal semigroups}\label{s4}
In general the semigroup generated by a convex body of $\R^2$ is not finitely generated. In this section partial results on semigroups generated by convex polygons are presented and the affine convex polygonal semigroups are characterized.

Denote by $P_i=(p_{i1},p_{i2})$ with $i=1,\ldots ,n$  the vertices of a compact convex polygon $F\subset \R^2_{\geq}$ ordered in the clockwise direction. We denote this set by  $\CaP$ and by $\FraP$ the associated semigroup.

\begin{proposition}\label{poligono_racional}
If $\CaP \subset \Q ^2_{\geq},$ then $\FraP$ is finitely generated. Furthermore, there exists an algorithm which determines its minimal system of generators.
\end{proposition}

\begin{proof}
Let $\CaP=\{P_1,\ldots ,P_n\}$ the set of vertices of $F$ and consider the set of points $\CaP'=\{(P_1,1),\ldots ,(P_n,1)\}\subset \Q_{\geq} ^3.$ Take now the cone $\FraC \subseteq \N ^3$ delimited by the planes that contain the origin and two consecutive points of $\CaP'.$ Since this cone is defined by rational inequalities, it is finitely generated.

A system of generators of $\FraP$ is the set formed by the projection onto the first two coordinates of a system of generators of  $\FraC.$ From this set of generators of $\FraP$ one can compute its minimal system of generators.
\end{proof}

Suppose now the extremal ray $\tau_1$ of $L_{\Q_{\geq}}(F)$ intersects $F$ in only one point $P_1,$ denote by $V_i$ the intersection of $\overline{(iP_1)(iP_2)}$ and  $\overline{((i+1)P_n)((i+1)P_1)}$ for every $i\in \N.$ Note that for the initial values of $i$ it is possible that these points does not exist (see Figure \ref{vertice_triangulo}).

\begin{lemma}\label{distancia_triangulos}
Every point $V_i$ belongs to a   parallel line to  $\tau _1.$
\end{lemma}

\begin{proof}
Clearly $\overline{(iP_1)(iP_2)}$ and  $\overline{((i+1)P_n)((i+1)P_1)}$ are not parallel, their lengths
increase with no limit and keep one of their vertices in the ray $\tau_1$. They intersect in only one point $V_i$ for $i\gg 0.$

After some basic computations the reader can check that the distance between  $V_i$ and $\tau_1$ is constant and equal to  $$ \left|{\frac {{p_{12}}^{2}p_{21}p_{n1}-p_{12}p_{21}p_{11}p_{n2}+{a_{
{1}}}^{2}p_{n2}p_{22}-p_{11}p_{22}p_{12}p_{n1}}{ \left( -p_{22}
p_{n1}+p_{11}p_{22}+p_{12}p_{n1}+p_{n2}p_{21}-p_{n2}p_{11}-b_
{{1}}p_{21} \right) \sqrt {{p_{12}}^{2}+{p_{11}}^{2}}}}\right|.$$
Thus, the points  $V_i$ are in a  line parallel to $\tau _1.$
\end{proof}

In this case there exists $i_0$ such that $$\int(\FraP) \setminus \cup_{i\geq 0}^{i_0}F_i \subset\int (\FraC) \setminus \cup_{i\geq i_0}{\rm{triangle}}({\{iP_1,(i+1)P_1,V_i\}}).$$
We illustrate this property in Figure \ref{vertice_triangulo} (in this figure $i_0=6$).
        \begin{figure}[h]
            \begin{center}
\includegraphics[scale=0.27]{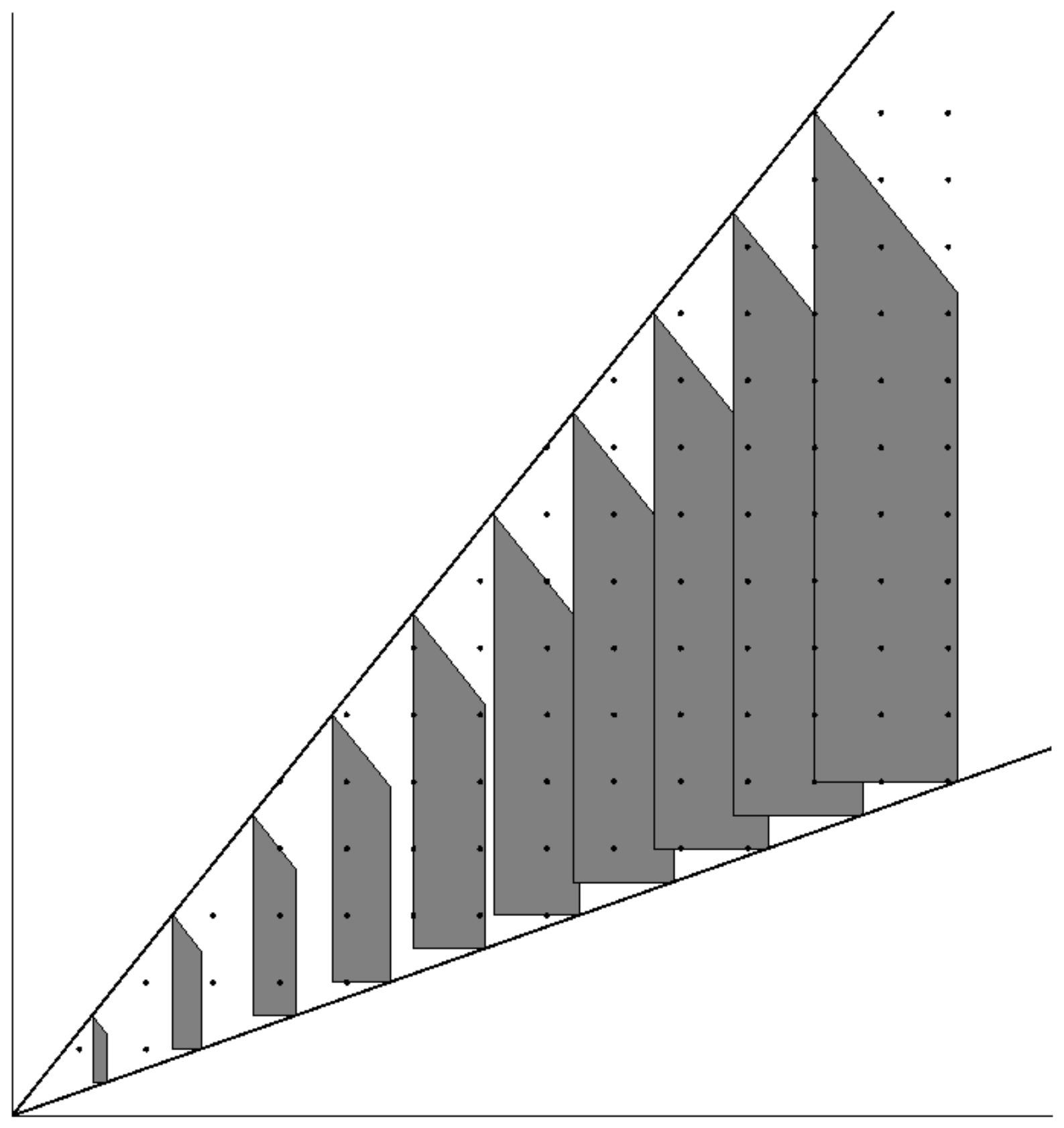}
            \caption{Image of a convex polygonal semigroup.}\label{vertice_triangulo}
            \end{center}
        \end{figure}

For the sake of simplicity we have used the points $P_1,$ $P_2$ and $P_n$ in the above results, but the result can be extended to the intersection of $F$ and an extremal ray when this intersection is only a point.

We focus now our attention when $F$ is a particular triangle.

\begin{proposition}\label{semigrupo_triangulo}
Let $F$ be a triangle delimited by $\{P_1,P_2,P_3\}$ with $P_1\in \Q^2_{\geq}$ and $P_2,P_3\in \R^2_{\geq} \setminus \Q^2,$ such that $P_1\in \tau_1$ and $\overline{P_2P_3}\subset \tau_2$, where $\tau_1$ and $\tau_2$ are the extremal rays of $L_{\Q_{\geq}}(F)$. Then $\FraP$ is finitely generated and there exists an algorithm to compute its minimal system of generators.
\end{proposition}

\begin{proof}
By Lemma \ref{distancia_triangulos}, for all integer $i\gg 0$ the distance between the point $i\overline{P_1P_2}\cap (i+1)\overline{P_1P_3}$ and the line $\tau_1$ is constant. Let $j_0$ be the smallest integer such that $j_0\overline{P_1P_2}\cap (j_0+1)\overline{P_1P_3}\neq \emptyset$ and $j_0P_1\in \N ^2.$ Denote by $s_1$ the element of $\FraP$ which generates $\FraP\cap \tau_1,$ by $V$ the point $j_0\overline{P_1P_2}\cap (j_0+1)\overline{P_1P_3},$ and let $j_1$ be the smallest integer such that $j_1P_1=j_0P_1+s_1.$

Denote by $T_1$ the finite set of integer points belonging to the parallelogram $G$ with edges the segment $\overline{(j_0P_1)(j_1P_1)}$ and the segment determined by  the points $j_0P_1$ and $j_0P_1+\overrightarrow{(j_0P_1)V}$, but they are not in $\FraP$. By Lemma \ref{lema_rectangulo},
the integer points of $G$ obtained applying the translations defined by  $is_1$ with $i\in \N$ are
the translated of  $T_1$. Furthermore, we clearly have the distances of the points of $T_1+is_1$ to the edges of the triangles contained in the parallelogram $G+is_1$ are constant for all $i\in \N.$

Denote by $T_2$ the finite set of integer points of the region delimited by $\tau_1,$ $\tau_2$ and $j_0\overline{P_1P_3}$ which does not belong to $\FraP,$ and let $T=T_1\cup T_2$(see Figure \ref{triangulo1}).
        \begin{figure}[h]
            \begin{center}
\includegraphics[scale=0.25]{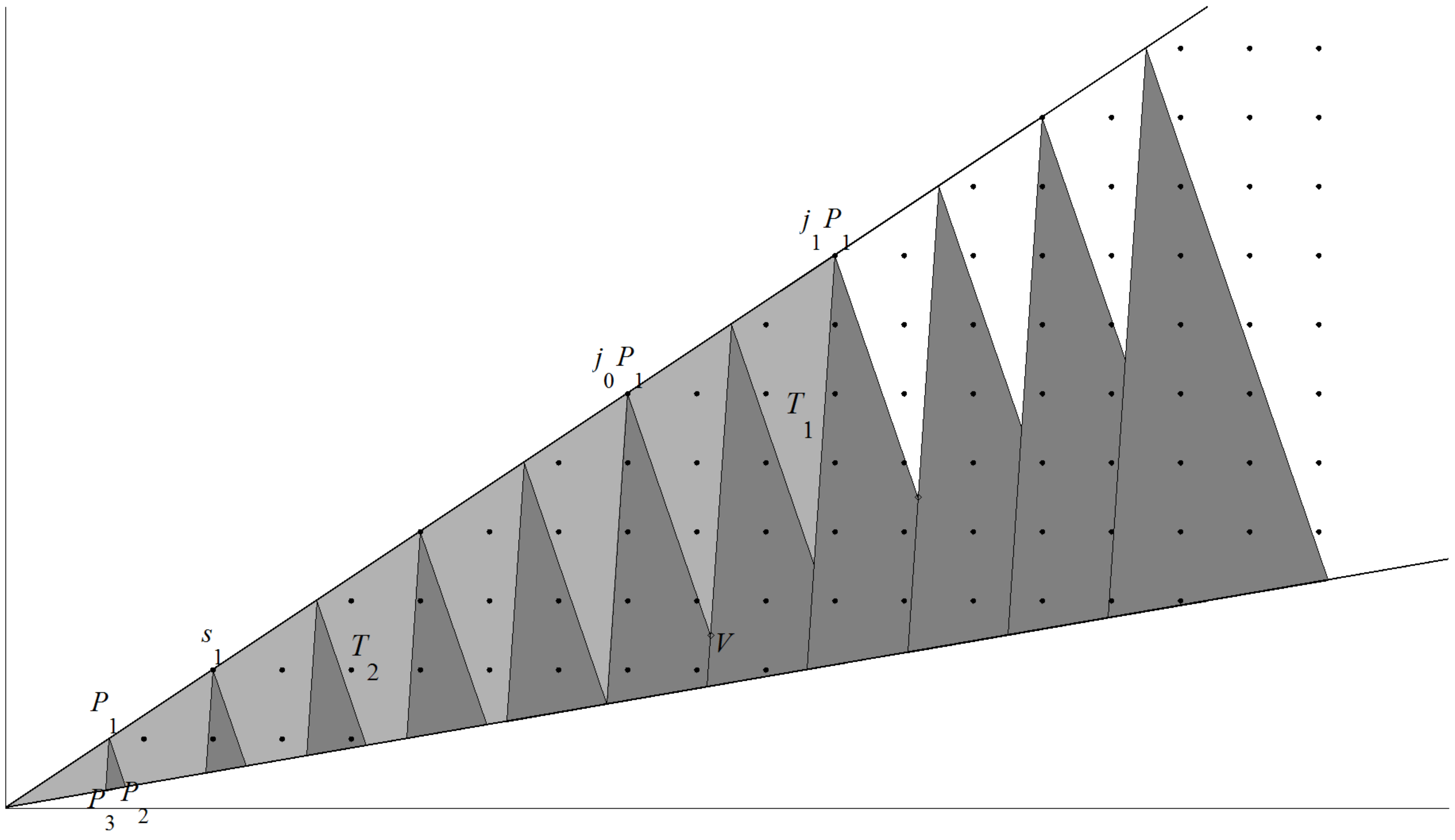}
            \caption{Set $T=T_1\cup T_2.$}\label{triangulo1}
            \end{center}
        \end{figure}

Consider $$d_1=\min\left( \bigcup_{i=1}^{j_1} \{d(H,i\overline{P_1P_2})|H\in T\}\right),$$ and $$d_2=\min \left(\bigcup_{i=0}^{j_1} \{d(H,(i+1)\overline{P_1P_3})|H\in T\}\right).$$

Once we know the distances $ d_1 $ and $ d_2 $ we can move in $ \tau_2 $ the vertices $ P_2 $  and $ P_3 $ until we reach two rational points $ P_2 '$ and $ P_3' $ (since the  slope of $ \tau_2 $ is rational, there are an infinite number of possibilities to take these points into segments that including $ \overline {P_2P_3} $) to form a new triangle $ F '$ with rational vertices $\{P_1, P_2', P_3 '\}$ such that 
$$\FraP=\left( \bigcup _{i\in \N} iF' \right) \cap \N^2,$$ as shown in Figure
\ref{triangulo2}, where dotted lines correspond to the new rational triangle with rational vertices.
        \begin{figure}[h]
            \begin{center}
\includegraphics[scale=0.25]{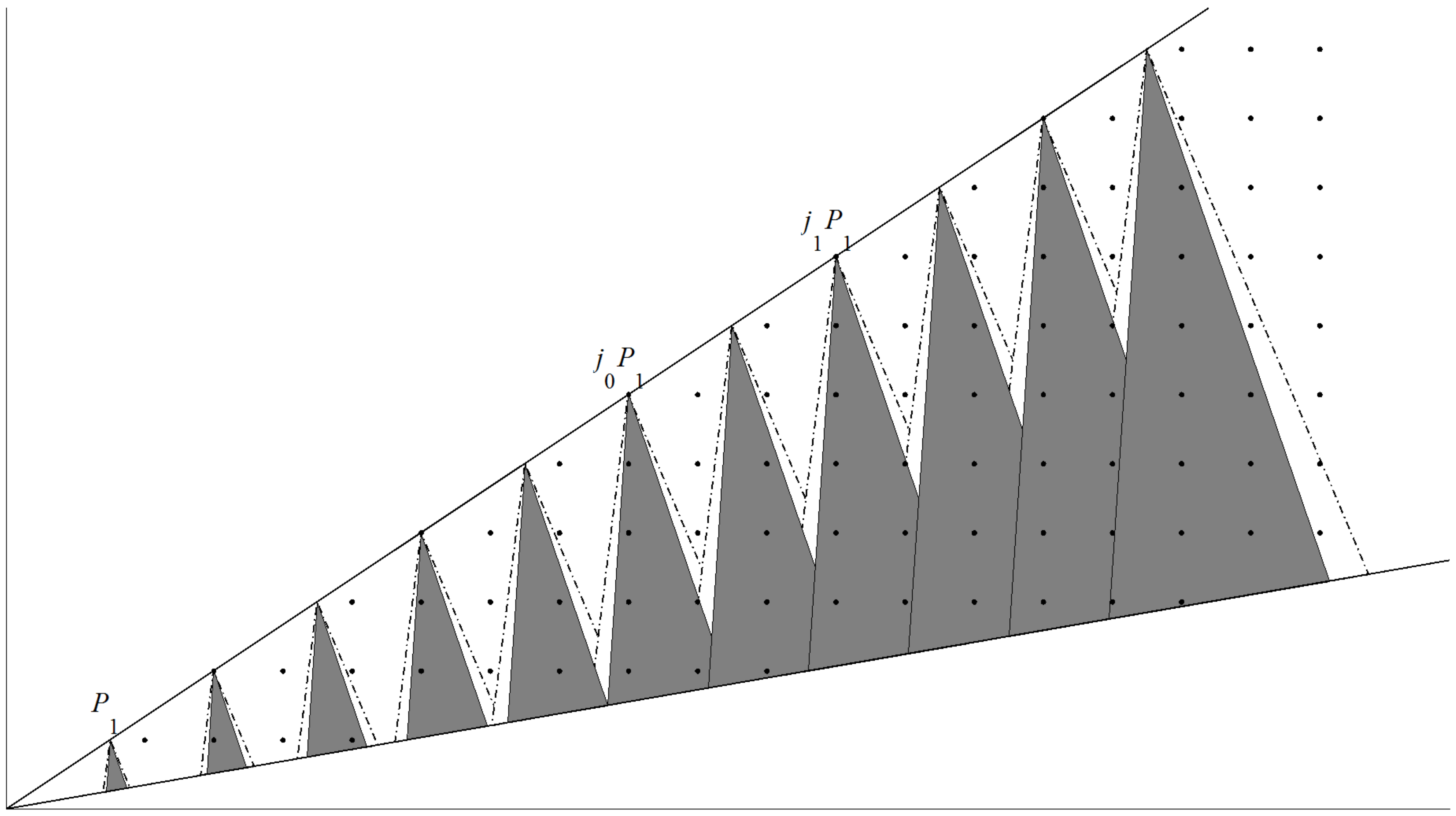}
            \caption{Construction of a triangle with rational vertices.}\label{triangulo2}
            \end{center}
        \end{figure}

As the vertices of $F'$ are rational, the semigroup $\FraP$ is finitely generated and its minimal system of generators can be computed (see Proposition \ref{poligono_racional}).
\end{proof}

\begin{proposition}\label{corte_segmentos}
Let $F\subset \R^2_{\geq}$ be a convex polygon fulfilling that   $\tau_1$ and $\tau_2$ have rational slopes and $\tau_1\cap F$ and $\tau_2\cap F$ are segments. Then $\FraP$ is finitely generated and there exists an algorithm which determines its minimal system of generators.
\end{proposition}

\begin{proof}
Let $\tau_1\cap F=\overline{P_1P_2}$ and $\tau_2\cap F=\overline{P_{l+1}P_l}.$ By construction there exists the least integer $j_0,$ such that the region $G$ bounded by $\tau_1,$ $\tau_2$ and the segment $j_0\overline{P_1P_{l+1}}$ verifies $\FraC\setminus G\subset \left( \bigcup _{i\ge j_0} iF \right) \cap \N^2.$ Define the finite set $T=G\setminus\FraP.$

Since $\FraP$ is the set  $\FraC\setminus T,$ we conclude that  $\FraP$ is finitely generated  (see Corollary \ref{semigrupo menos un numero finito de puntos}).

An algorithm to determine a system of generators of  $\FraP$ is the following:
\begin{enumerate}
\item Compute the generators of $\FraP\cap \tau_1$ and $\FraP\cap \tau_2$ (use Lemma \ref{generadores_rayos}).
\item Construct a semigroup $\FraF'$ verifying $\FraF'\cap \tau_1=\FraP \cap \tau _1,$ $\FraF'\cap \tau_2=\FraP \cap \tau _2$ and $\FraF'\setminus\{\tau_1,\tau_2\}=\FraC\setminus\{\tau_1,\tau_2\}$ (use Lemma \ref{sg_S_prima2}). This semigroup is obtained using the system of generators of $\FraC$ and the generators set of the preceding step.
\item Eliminate from $\FraF'$ all the points of  $T$ (use Lemma \ref{quitando elementos}).
\end{enumerate}
This process ends after a finite number of steps obtaining  a system of generators of $\FraP$ which can used to get its minimal system of generators.
\end{proof}

\begin{theorem}\label{teorema_poligonos_fg}
The semigroup $\FraP$ is finitely generated if and only if  $F\cap\tau_1$ and $F\cap\tau_2$ contain rational points. Furthermore, in such case there exists an algorithm to compute the minimal system of generators of $\FraP.$
\end{theorem}

\begin{proof}
Assume $F\cap \tau_1\subseteq \R^2_{\geq}\setminus\Q^2$ and let $G=\{s_1,s_2,\ldots, s_r\}$ be a system of generators of  $\FraP.$ This implies that $\FraP\cap \tau_1=\emptyset.$

Consider $s_k\in G$ such that the vector $\overrightarrow{Os_k}$ has maximum slope respect to the points of $G$. Since $\FraP \cap\tau_1=\emptyset$, there exists at least an element $Q\in\Q^2$ in the interior of the cone delimited by $\tau_1$ and the ray defined by $s_k$.

There exists $u\in \N$ such that  $uQ$ belongs to a polygon $F_{i_0}$, but  $uQ$ is not generated by $G.$ Thus, $\FraP$ is not finitely generated which is a contradiction.

If $F\cap \tau_2$ has not rational points, the proof that $\FraP$ is not finitely generated is similar than above.

Conversely, assume the intersections of $F$ with $\tau_1$ and $\tau_2$ contain rational points. There are several cases:
\begin{enumerate}
\item If $\tau_1\cap F$ and $\tau_2\cap F$ are segments, this case is already solved in Proposition \ref{corte_segmentos}.
\item If $\tau_1\cap F$  has only a point and $\tau_2\cap F$ is a segment, then take  $\tau_1'$ a ray  with rational slope such that the intersection of the polygon $F$ with the region delimited by $\tau_1$ and $\tau_1'$ is a triangle $F_1'.$ The set $F_2'=F\setminus F_1'$ verifies the conditions of Proposition \ref{corte_segmentos}.

    The minimal system of generators of the semigroup generated by  $F_1'$ can be computed  (use Proposition \ref{semigrupo_triangulo}).
Analogously, the minimal system of generators of the semigroup generated by $F_2'$ can be computed (use Proposition  \ref{corte_segmentos}). Since  $\FraP$ is the union of the semigroups generated by  $F_1'$ and $F_2',$ the semigroup $\FraP$ is finitely generated by the union of the above systems of generators.
\item If $\tau_1\cap F$ and $\tau_2\cap F$ are two points, we proceed as follows. Take $\tau_1'$ and $\tau_2'$ two rays with rational slopes such that the polygons obtained from the intersection of $F$ and the region delimited by  $\tau_1$ and $\tau_1',$ and by $\tau_2$ and $\tau_2',$ are two triangles. The intersection of the polygon $F$ and the region delimited by $\tau_1'$ and $\tau_2'$ verifies the condition of Proposition \ref{corte_segmentos} (see Figure \ref{poligono_cualquiera}).
        \begin{figure}[h]
            \begin{center}
\includegraphics[scale=0.2]{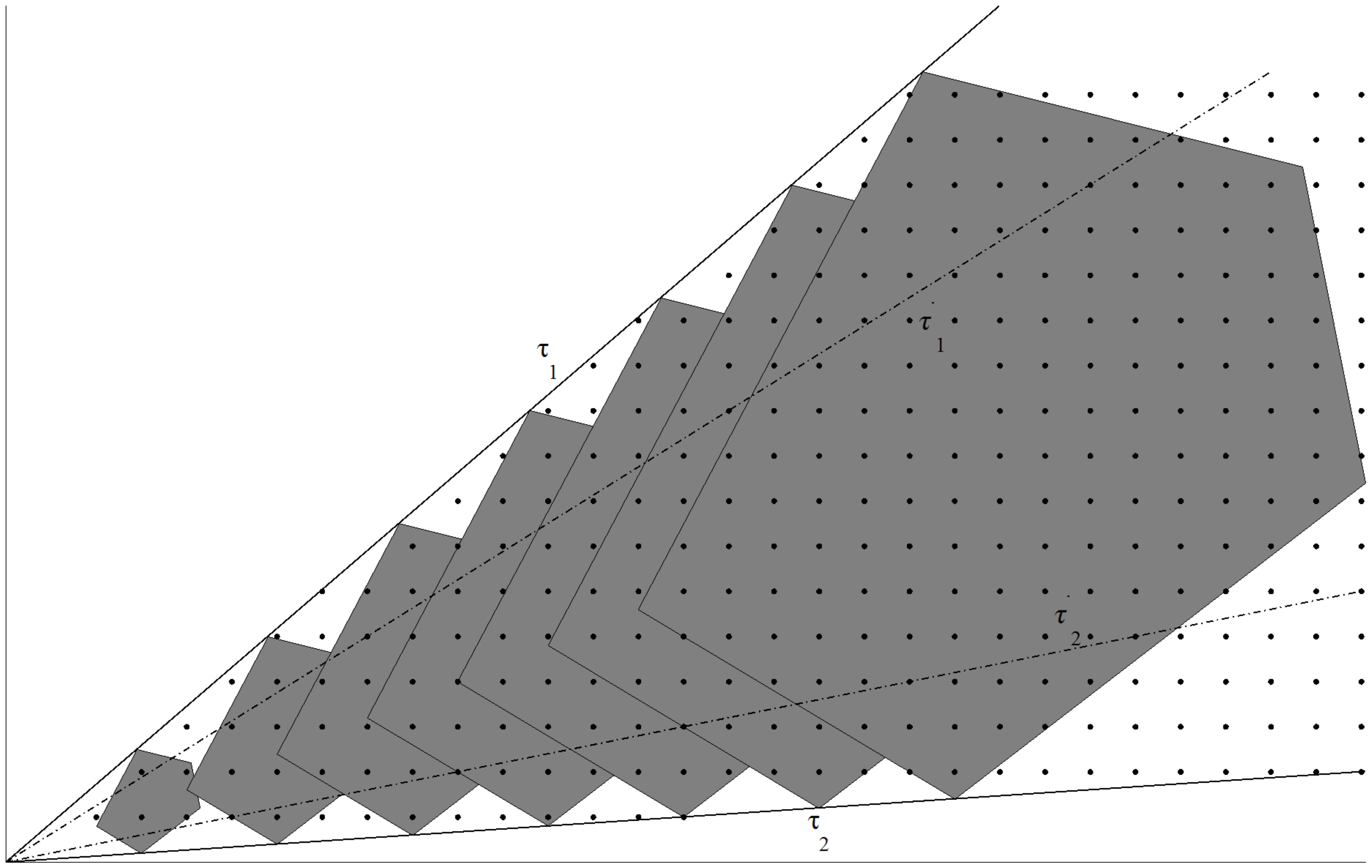}
            \caption{Polygon with only a vertex in each extremal rays.}\label{poligono_cualquiera}
            \end{center}
        \end{figure}

Once again, a system of generators of $\FraP$ can be obtained by applying Proposition \ref{semigrupo_triangulo} and Proposition \ref{corte_segmentos} to the above regions.
\end{enumerate}
In any case the semigroup $\FraP$ is finitely generated and its minimal system of generated can be computed algorithmically.
\end{proof}

\section{Circle semigroups}\label{s5}
This section is devoted to semigroups generated by circles. The reason of this section is that most of the results of Section \ref{s4} are not valid for this kind of semigroups.

Let $C$ be the circle (a convex body) with center  $(a,b)$ and radius $r.$ Denote by $C_i$ the circle with center $(ia,ib)$ and radius $ir,$ and by $\FraS=\bigcup_{i=0}^\infty C_i\cap \N^2$ the semigroup generated by $C.$ As in the preceding sections, denote by  $\tau_1$ and $\tau_2$ the extremal rays of $L_{\Q_{\geq}}(C\cap \R ^2_{\geq})$ where the slope of $\tau_1$ is greater than the slope of  $\tau_2$, and by  $\FraC$ the positive integer cone $L_{\Q_{\geq}}(C\cap \R ^2_{\geq})\cap \N^2.$ In such case, $\int(\FraC)=\FraC\setminus\{\tau_1,\tau_2\}.$

\begin{lemma}\label{limite_distancia}
Suppose that $C\cap \tau_2$ is a point. If $P_i$  is the closest point to $\tau_2$ belonging to  $C_i\cap C_{i+1}$ \footnote{For the initial values of $i$ it is possible to obtain that $C_i\cap C_{i+1}=\emptyset$, see Figure \ref{h2}.}, then $\displaystyle{\lim_{i\to \infty} \d(P_i,\tau_2)=0}.$
\end{lemma}

\begin{proof}
Denote by $h_i$ the distance $\d(P_i,\tau_2).$ Without lost of generality, assume that $\tau_2$ is the line $\{y=0\}.$ This is possible because the distances between the points of our construction are invariant by turn centered in the origin. Graphically the situation is as shown in Figure \ref{h2}.
        \begin{figure}[h]
            \begin{center}
\includegraphics[scale=0.4]{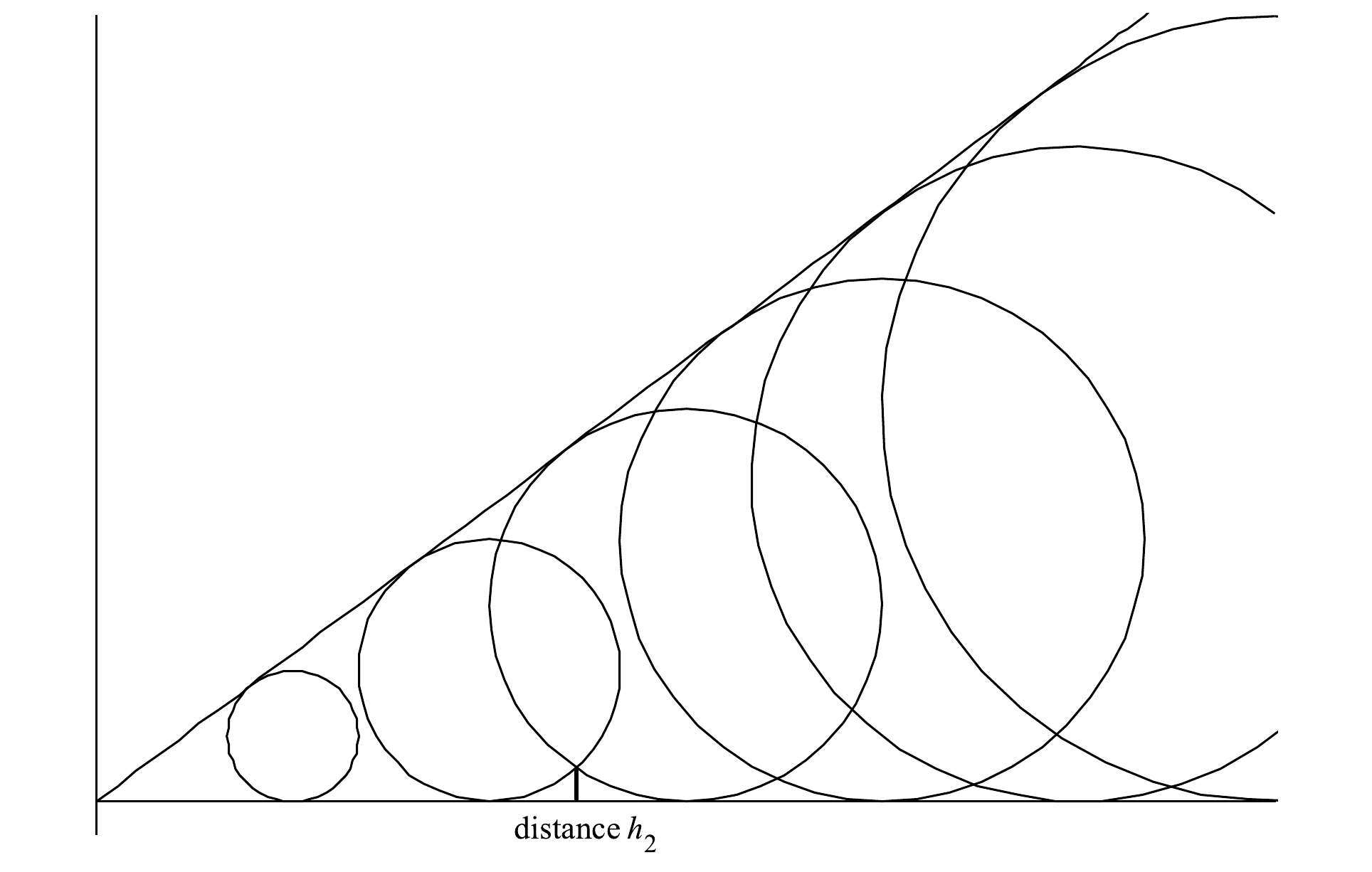}
            \caption{Distance $h_2.$}\label{h2}
            \end{center}
        \end{figure}

Since the slope of $\tau_2$ is zero, the circles have radius $ib$ and therefore  $h_i=\d(P_i,\tau_2)$ is equal to the second coordinate of $P_i.$

With these hypothesis, the point $P_i$ is the solution of the following system of equations closest to the axis $OX$
        $$\left\{\begin{array}{ccccc}
        C_i & \equiv & (x-a i)^2+(y -b i)^2 & = & (b i)^2\\ \\
        C_{i+1} & \equiv & \Big(x-a (i+1)\Big)^2+\Big(y -b (i+1)\Big)^2 & = & b^2 (i+1)^2.
        \end{array}\right.$$
That is,
$$
x=\frac{a^4 (1+2 i)+b \sqrt{-a^4 \left(a^2-4 b^2 i (1+i)\right)}}{2 a \left(a^2+b^2\right)},$$
$$y= \frac{a^2 (b+2 b i)-\sqrt{-a^4 \left(a^2-4 b^2 i (1+i)\right)}}{2 \left(a^2+b^2\right)}.
$$Then the distance is \begin{equation}\label{distancia_hi}h_i=\d(P_i,\tau_2)=\frac{a^2 b+2 a^2 b i-\sqrt{-a^6+4 a^4 b^2 i+4 a^4 b^2 i^2}}{2
\left(a^2+b^2\right)}.\end{equation}
It is straightforward to prove that $\displaystyle{\lim_{i\to \infty} h_i =0}.$
\end{proof}

\begin{remark}
If $C\cap \tau_1$ has only a point, denote by $P'_i$ the point of $C_i\cap C_{i+1}$ closest to $\tau_1.$ Using the symmetry of $\bigcup_{i= 0}^{\infty} C_i$ with respect to the line joining  the centers of the circles, we obtain that $\d(P_i',\tau_1)=\d(P_i,\tau_2).$
\end{remark}

The following Lemma asserts that  $\int(\FraC) \setminus \int (\FraS)$ has a finite number of points if $C\subset \R^2_{\geq}$.

\begin{lemma}\label{lema_interior_cono}
Let $C\subset \R^2_{\geq}$ be a circle. There exists $d\in \R _{\geq}$ such that $$\{P\in \int(\FraC) | \d(P)>d\}\subset \FraS.$$
Furthermore, $d$ can be computed algorithmically.
\end{lemma}

\begin{proof}
Consider two rectangles in  $\FraC$ whose bases are segments determined by two consecutive points of the semigroup in $\tau_1$ for the firsts rectangle and in $\tau_2$ for the second  and with height (the same for both) a sufficiently small value to obtain no points of $\N^2$ into them (excepting in their bases). Denote by $d'$ this height.

For the sake of simplicity we consider that $\tau_2$ is the line $\{y=0\}.$ In this case the rectangles are as in Figure \ref{rectangulos}.
              \begin{figure}[h]
            \begin{center}
\includegraphics[scale=0.35]{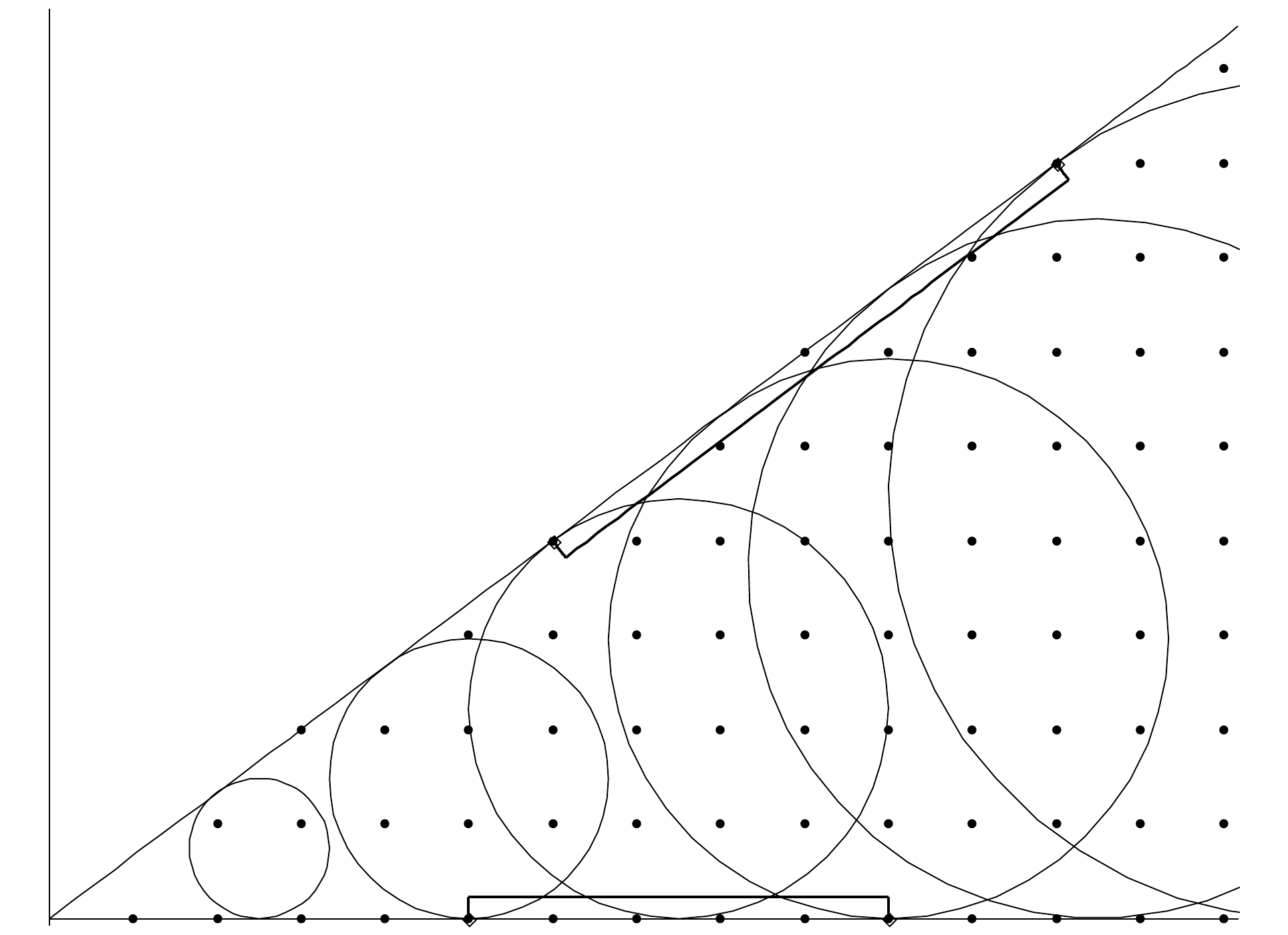}
            \caption{Construction 1.}\label{rectangulos}
            \end{center}
          \end{figure}

Denote by  $T_1,T_2\in \FraS$ the vertices\footnote{Note the point $T_1$ is a natural multiple of the point $\tau_2\cap C$ and that $T_2=2T_1.$} of the base of the rectangle over the line $\tau_2.$

Consider now the region of the cone obtained applying to the above rectangle all the translations defined by the vector  $\overrightarrow{OT_1}$ and all its positive multiples. This construction is done over $\tau_1$ and over  $\tau_2$ (see Figure \ref{franjas}). In this region there are not integer points (Lemma  \ref{lema_rectangulo}).
        \begin{figure}[h]
            \begin{center}
\includegraphics[scale=0.39]{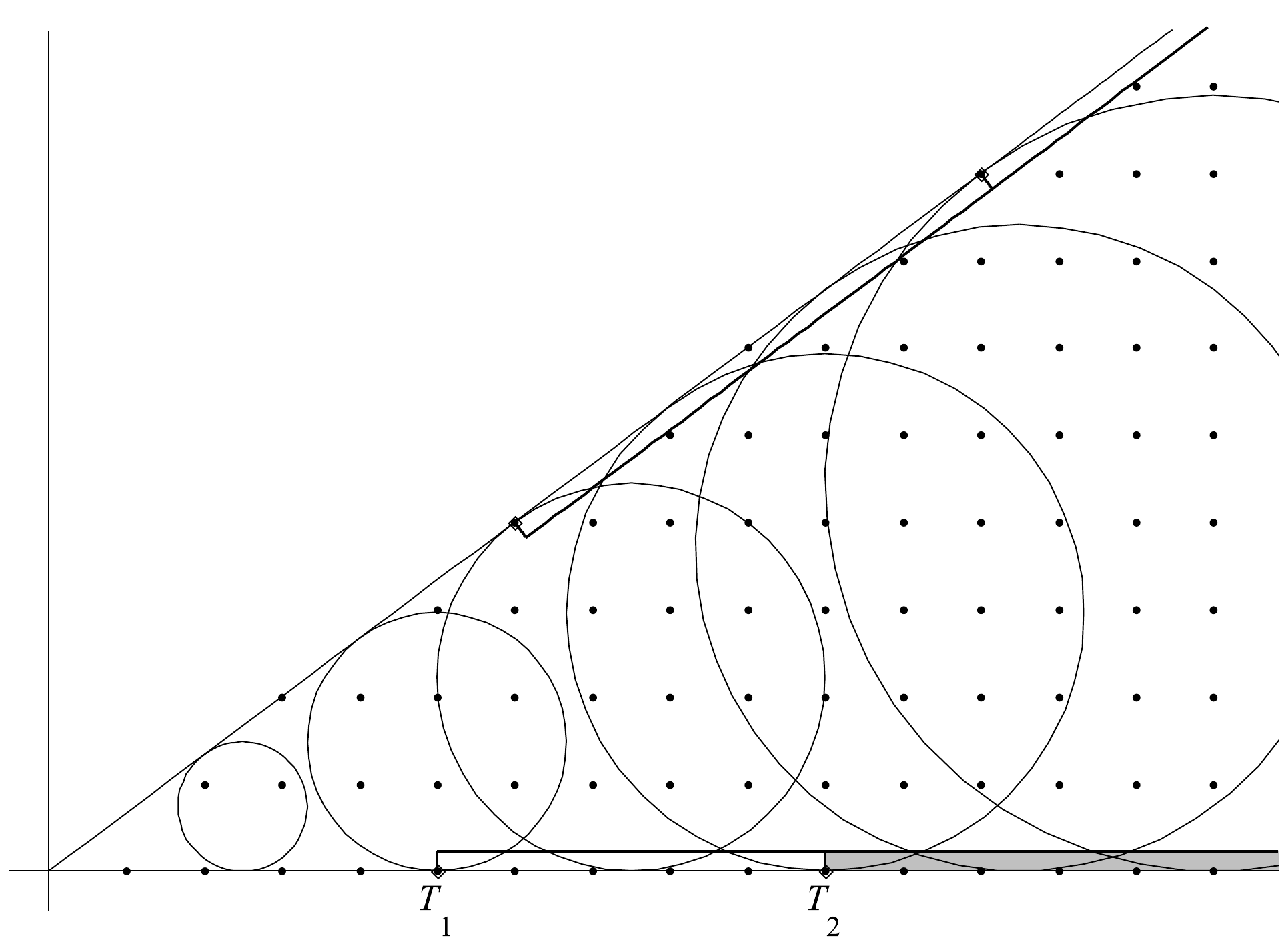}
            \caption{Construction 2.}\label{franjas}
            \end{center}
        \end{figure}

Let $i_0\in \N$ the first term of the sequence of heights  $\{h_i\}_i$ (defined in (\ref{distancia_hi})) such that $h_{i_0}<d'.$ Lemma \ref{limite_distancia} asserts the existence of  $i_0.$

Then there exists  $d \in \R_{\geq}$ determined by the circle  $C_{i_0}$ such that  $\{P\in \int(\FraC) | \d(P)>d\}\subset \bigcup_{i\ge i_0}^{\infty} C_i \cap \N^2 \subset \FraS.$ In Figure  \ref{interior_cono}, observe  that $i_0=6$.
        \begin{figure}[h]
            \begin{center}
\includegraphics[scale=0.4]{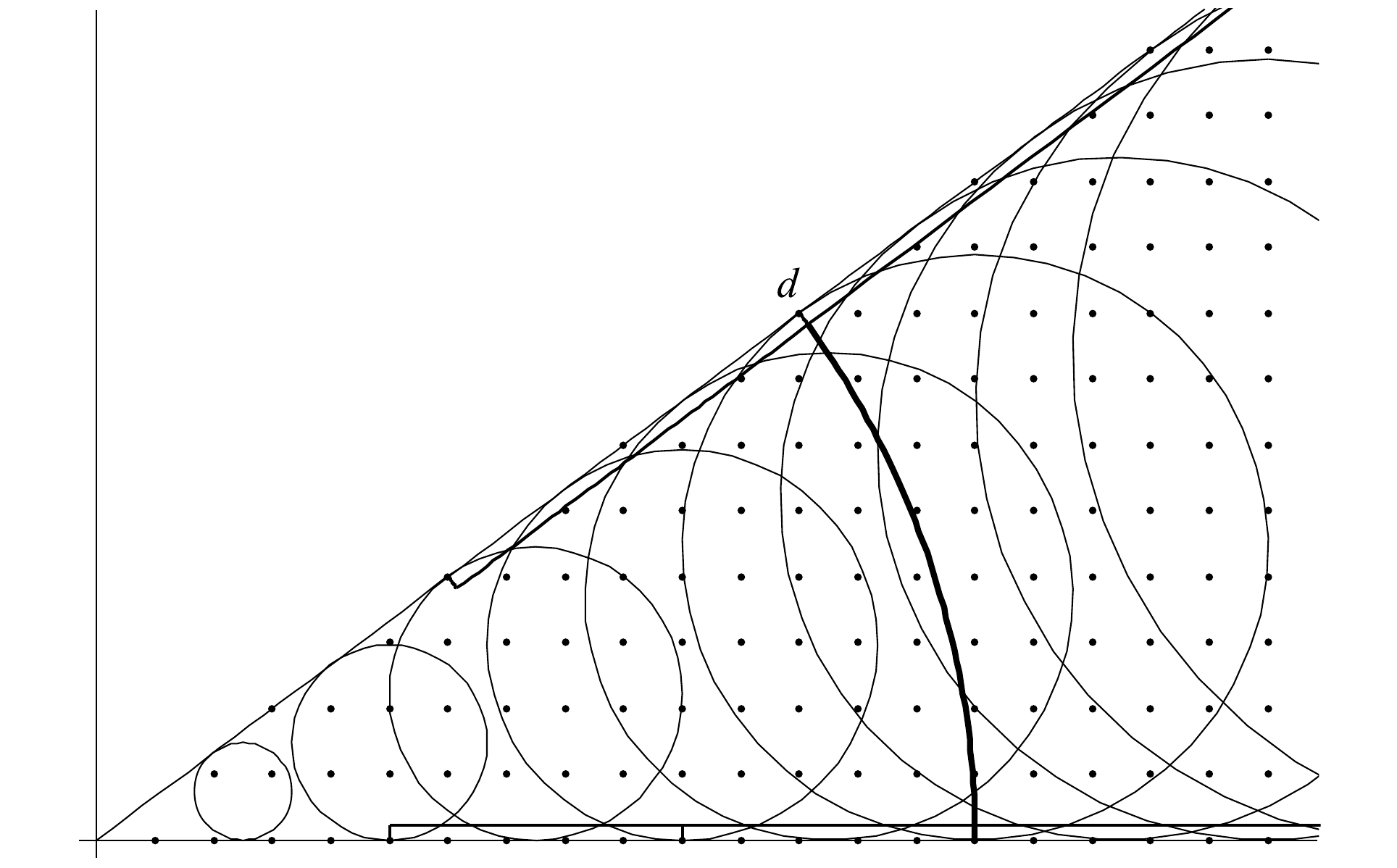}
            \caption{Construction 3.}\label{interior_cono}
            \end{center}
        \end{figure}
\end{proof}

The region delimited by $\tau_1$, $\tau_2$ and the circle with center the origin and radius  $d$ of the above lemma  (Figure \ref{interior_cono}) can be replaced by the triangle delimited by the $\tau_1$, $\tau_2$ and the line joining the points of the intersection of such lines with the circle $C_{i_0}$. This simplifies the computation of the integer points of the region.

The following Theorem characterizes  affine circle semigroups and provides an algorithm to compute their minimal system of generators.

\begin{theorem}\label{teorema_circulos_fg}
The semigroup $\FraS$ is finitely generated if and only if $C\cap \tau_1$ and $C\cap \tau_2$ have rational points. Furthermore, in such case the minimal system of generators of  $\FraS$ can be computed algorithmically.
\end{theorem}

\begin{proof}
If $\FraS$ is finitely generated proceed as in Theorem \ref{teorema_poligonos_fg}.

For the reciprocal we consider several cases. If $C\cap \R^2_{\geq}=\emptyset,$ then $\FraS=\{0\}$ and therefore it is finitely generated. In other case, compare the semigroups $\FraS$ and $\FraC$. The relationship between the sets $\int(\FraC)$ and $\int(\FraS)$ is the following:
if $P\in \int(\FraC)\setminus \int(\FraS)$ then $\d(P)\le d,$ where $d$ is the distance determined by Lemma \ref{lema_interior_cono}. Therefore $\int(\FraC)\setminus \int(\FraS)$ is finite. In addition, given $P\in \N^2$ with $\d(P)>d,$ $P\in \int(\FraC)$ if and only if $P\in \int(\FraS).$

To study the relationship between $\FraC\cap \tau_1$ and $\FraS\cap \tau_1,$ and  $\FraC\cap \tau_2$ and $\FraS\cap \tau_2,$ we must consider four cases:
\begin{enumerate}
\item Assume that $C\cap \tau_1$ and $C\cap \tau_2$ have only one point (this situation is similar to that shown in Figure \ref{interior_cono}).
In this case, if $\FraC\cap \tau_1=\langle g_1 \rangle$ and $\FraC\cap \tau_2=\langle g_2 \rangle,$ then all the elements of $\FraS\cap \tau_1$ and $\FraS\cap \tau_2$ are natural multiples of $g_1$ or $g_2.$

\item Assume that $C\cap \tau_1$ is a point and $C\cap \tau_2$ is a segment. In this case $\tau _2$ is the line $\{y=0\}$ (see Figure \ref{pendiente_negativa}).
        \begin{figure}[h]
            \begin{center}
\includegraphics[scale=0.4]{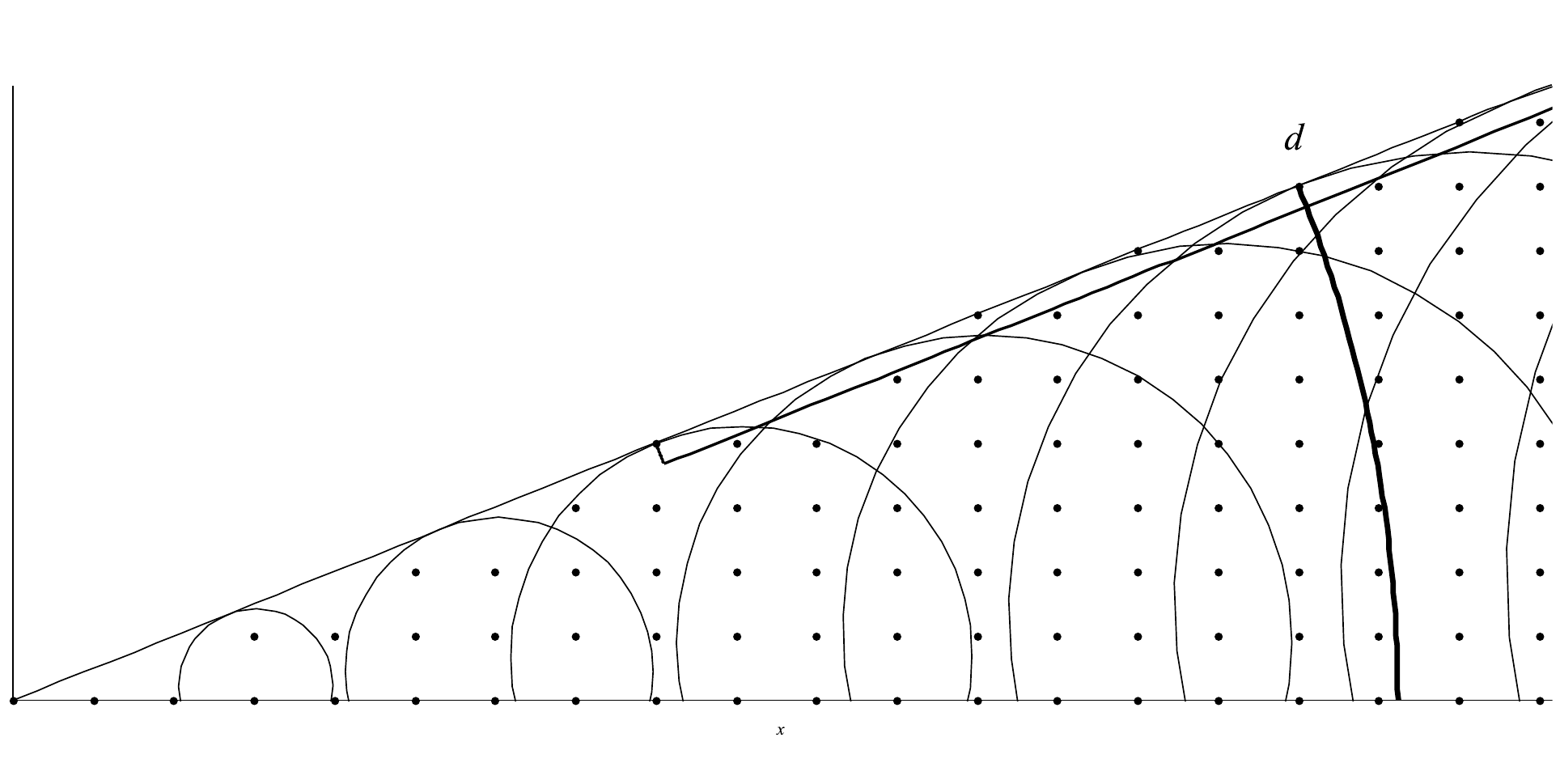}
            \caption{$C\cap \tau_2$ is a segment.}\label{pendiente_negativa}
            \end{center}
        \end{figure}
We compare again the semigroups $\FraS$ and $\FraC$:
\begin{itemize}
\item Note that if $\FraC\cap \tau_1=\langle g_1 \rangle$ then all the elements of $\FraS\cap \tau_1$ are natural multiples of $g_1.$
\item The set $(\FraC\cap \tau_2)\setminus (\FraS\cap \tau_2)$ is finite. Besides,  $\FraS\cap \tau_2$ is a finitely generated semigroup and its minimal system of generators can be computed algorithmically (see Lemma \ref{generadores_rayos}).
\end{itemize}

\item Assume that $C\cap \tau_1$ is a segment and $C\cap \tau_2$ is a point. This case is similar to the above case.

\item Assume that $C\cap \tau_1$ and $C\cap \tau_2$ are  segments. In this case $\tau _1$ is the line $\{x=0\}$ and $\tau _2$ is the line $\{y=0\}.$
Then the sets $(\FraC\cap \tau_1)\setminus (\FraS\cap \tau_1)$ and $(\FraC\cap \tau_2)\setminus (\FraS\cap \tau_2)$ are finite. Besides,  $\FraS\cap \tau_1$ and $\FraS\cap \tau_2$  are two finitely generated semigroups and their minimal systems of generators can be computed algorithmically (see Lemma \ref{generadores_rayos}).
\end{enumerate}

We have obtained that in any case $\FraS$ is the set obtained after eliminate from  $\FraC$ a finite number of points of its interior and some points of its extremal rays.

See now how a system of generators of $\FraS$ can be built. We construct explicitly a set of generators of the semigroup $\FraS'$ such that $\FraS'\cap \tau_1=\FraS\cap \tau_1,$ $\FraS'\cap \tau_2=\FraS\cap \tau_2,$ and $\int(\FraS')=\int(\FraC).$ This set will be used in Corollary \ref{bound}.

Denote by  $\{g_1, \ldots, g_p\}$ the minimal system of generators of  $\FraC$ where  $g_1\in  \tau_1$ and $g_2\in \tau_2.$
If we consider the first case and assume that $s_1$ and $s_2$ are the minimal elements of $\FraS$ in $\tau_1$ and $\tau_2,$ then there exist $k_1,k_2\in \N$ such that $s_1=k_1g_1$ and $s_2=k_2g_2.$ By using Lemma \ref{sg_S_prima2} on $s_1$ and after on $s_2,$ the semigroup  $\FraS'$ is generated by
\begin{equation}\label{gen_rayos1}
\{s_1,s_2,g_3,\ldots ,g_p\}\cup  \left(\bigcup _{i=2}^p
\{g_i+g_1, \ldots, g_i+(k_1-1)g_1\}\right) \cup \{s_1+g_2, \ldots, s_1+(k_2-1)g_2\} \cup$$
$$\cup
\bigcup_{j=1}^{k_2-1} \left(\bigcup _{i=2}^p \{g_i+g_1+jg_2, \ldots, g_i+(k_1-1)g_1+jg_2\}\right)
 \cup
\bigcup _{i=3}^p \{g_i+g_2, \ldots, g_i+(k_2-1)g_2\}.
\end{equation}

Consider now the second case (analogously for the third case). There exists $k_1\in \N$ such that  $s_1=k_1g_1\in \tau_1,$ and there exist $\lambda _1,\ldots , \lambda _t\in \N$ such that  $\lambda _1< \cdots < \lambda _t$ and $\FraS\cap \tau_2$ is generated minimally by  $\{(\lambda _i,0)=\lambda_i (1,0)| i=1\ldots ,t\}$ ($g_2=(1,0)$). By using Lemma \ref{sg_S_prima2}, one obtain a system of generators of the semigroup  $\FraS'$,
\begin{equation}\label{gen_rayos2}\{s_1,g_3,\ldots ,g_p\}\cup \left(\bigcup_{i=1}^t \{ \lambda _i g_2 \}\right) \cup \{ s_1+ g_2,\ldots ,s_1+(\lambda_t-1)g_2 \}  \cup  $$
$$
\cup \left(\bigcup _{i=3}^p \{g_i+g_2, \ldots, g_i+(\lambda_t-1)g_2\}\right)
\cup \left(\bigcup _{i=2}^p \{g_i+g_1, \ldots, g_i+(k_1-1)g_1\}\right)
$$
$$
\cup
\bigcup_{j=1}^{\lambda_t-1} \left(\bigcup _{i=2}^p \{g_i+g_1+jg_2, \ldots, g_i+(k_1-1)g_1+jg_2\}\right)
\end{equation}

For the fourth case, there exist $\lambda _1,\ldots , \lambda _t,\lambda' _1,\ldots , \lambda' _{t'}\in \N$ such that $\lambda _1< \cdots < \lambda _t$ and $\lambda' _1< \cdots < \lambda' _{t'},$ $\FraS\cap \tau_1$ is generated minimally by  $\{(0,\lambda' _i)=\lambda'_i (0,1)| i=1\ldots ,t'\}$ ($g_1=(0,1)$) and $\FraS\cap \tau_2$ is generated minimally by  $\{(\lambda _i,0)=\lambda_i (1,0)| i=1\ldots ,t\}$ ($g_2=(1,0)$). Then $\FraS'$ is generated by
\begin{equation}\label{gen_rayos3}
\left(\bigcup _{i=1}^{t'} \{\lambda' _ig_1\}\right)\cup \left(\bigcup_{i=1}^t \{ \lambda _i g_2 \}\right) \cup \{g_3,\ldots ,g_p\} \cup \left(\bigcup_{i=1}^{t'}\{ \lambda' _i g_1+ g_2,\ldots ,\lambda' _ig_1+(\lambda_t-1)g_2 \}\right)  \cup  $$
$$
\cup \left(\bigcup _{i=3}^p \{g_i+g_2, \ldots, g_i+(\lambda_t-1)g_2\}\right)
\cup
\bigcup_{j=1}^{\lambda_t-1} \left(\bigcup _{i=2}^p \{g_i+g_1+jg_2, \ldots, g_i+(\lambda '_{t'}-1)g_1+jg_2\}\right).
\end{equation}

In any case, $\FraS'\cap \tau_1=\FraS\cap \tau_1,$ $\FraS'\cap \tau_2=\FraS\cap \tau_2,$ and $\int(\FraS')=\int(\FraC).$ Besides, $\FraS \subseteq \FraS'$ and $\FraS'\setminus \FraS$ is finite (if $P\in \FraS'\setminus \FraS,$ then $\d(P)\le d$).

Therefore, by Corollary \ref{semigrupo menos un numero finito de puntos}, $\FraS=\FraS'\setminus\left( \FraS'\setminus \FraS\right)$ is finitely generated. Moreover, a system of generators of $\FraS$ can be computed from a system of generators of $\FraS'.$ The idea of the algorithm is to eliminate from the minimal system of generators of $\FraS'$ the finite set of element $\FraS'\setminus \FraS$ by using the algorithm shown in Corollary \ref{semigrupo menos un numero finito de puntos}. At the end of this process  the minimal system of generators of $\FraS$ is obtained.
\end{proof}

The following Lemma allows to check if an element belongs to the semigroup  $\FraS$ by using its distance to the origin.

\begin{lemma}\label{pertenencia_circulo}
Let $(x,y)\in \N^2.$ The element $(x,y)\in \FraS$ if and only if $(x,y) \in C_k\cup C_{k+1}$ with  $k= \left\lfloor \sqrt{\displaystyle{\frac{x^2+y^2}{a^2+b^2}}} \right\rfloor\in \N.$
\end{lemma}

\begin{proof}
Given $(x,y)\in \FraS,$ the following inequalities holds
$$kd((a,b))\leq d((x,y))\leq (k+1)d((a,b)), $$ where $k=\left\lfloor \sqrt{\displaystyle{\frac{x^2+y^2}{a^2+b^2}}} \right\rfloor.$ Then $(x,y)$ belongs to $C_k$ and/or to $C_{k+1}.$
\end{proof}

Thus, to detect if an element is in  $\FraS$, it is enough to compare its distance to the origin  with the distance to the center of $C$.
After that, it only remains to check if the point belongs to two circles of  $\FraS$.

In the following result, Proposition \ref{ecuacion} is used to  obtain several inequalities satisfied by the elements of  $\FraS$.
\begin{corollary}\label{ecuacion_circulos}
Every $X=(x,y)\in \FraS\setminus\{\tau_1,\tau_2\}$ satisfies
$$\frac{1}{2}\left(\frac{(a,b)\cdot (x,y)}{{\sqrt{(\d(X) r)^2 -[(b,-a)\cdot (x,y)]^2}}}+1\right)
    \d(X)
    \mod \frac{\d(X) \left(\d((a,b))^2-r^2\right)}{{2\sqrt{(\d(X) r)^2 -[(b,-a)\cdot (x,y)]^2}}}
    \le \d(X).$$
\end{corollary}

\begin{proof}
Repeating the reasonings of Proposition  \ref{ecuacion} and Corollary \ref{necesaria_y_suficiente_ecuacion}, the coefficients of the inequality (\ref{ecuacion_inicial}) are determined by the points of the intersection of $C$ and the ray given by $X.$

In this case, the points are
$$P =\left(\frac{ {x} \left( {a}  {x} + {b}  {y}-\sqrt{- \left( {b}^2  {x}^2-2  {a}  {b}  {x}  {y}+ {a}^2  {y}^2-\left( {x}^2+ {y}^2\right) r^2\right)}\right)}{{x}^2+ {y}^2}\right.,$$
$$\left. \frac{ {a}  {x}  {y}+ {b}  {y}^2-\sqrt{- {y}^2 \left( {b}^2  {x}^2-2  {a}  {b}  {x}  {y}+ {a}^2  {y}^2-\left( {x}^2+ {y}^2\right) r^2\right)}}{ {x}^2+ {y}^2}\right),$$
$$Q=\left(\frac{ {x} \left( {a}  {x} + {b}  {y}+\sqrt{-\left( {b}^2  {x}^2-2  {a}  {b}  {x}  {y}+ {a}^2  {y}^2-\left( {x}^2+ {y}^2\right) r^2\right)}\right)}{{x}^2+ {y}^2}\right.,$$
$$\left. \frac{ {a}  {x}  {y}+ {b}  {y}^2+\sqrt{- {y}^2 \left( {b}^2  {x}^2-2  {a}  {b}  {x}  {y}+ {a}^2  {y}^2-\left( {x}^2+ {y}^2\right) r^2\right)}}{ {x}^2+ {y}^2}\right),$$
and
$$\d(P)= \frac{ {a}  {x} + {b}  {y}-\sqrt{-( {b}  {x}- {a}  {y})^2+\left( {x}^2+ {y}^2\right) r^2}}
{\sqrt{ {x}^2+ {y}^2}},$$
$$\d(Q)= \frac{ {a}  {x} + {b}  {y}+\sqrt{ -( {b}  {x}- {a}  {y})^2+\left( {x}^2+ {y}^2\right) r^2}}
{ \sqrt{ {x}^2+ {y}^2}}.$$

By Corollary  \ref{necesaria_y_suficiente_ecuacion}, $\d(X)$ verifies the inequality  $$\frac{\d(Q)}{\d(Q) - \d(P)} \d(X) \mod \frac{\d(Q)  \d(P)}{\d(Q) - \d(P)} \le \d(X),$$ where $$\frac{\d(Q)  }{\d(Q) - \d(P)}= \frac{1}{2}\left(\frac{(a,b)\cdot (x,y)}{{\sqrt{(\d(X) r)^2 -[(b,-a)\cdot (x,y)]^2}}}+1\right)$$ and $$\frac{\d(Q)  \d(P)}{\d(Q) - \d(P)}= \frac{\d(X) \left(\d((a,b))^2-r^2\right)}{{2\sqrt{(\d(X) r)^2 -[(b,-a)\cdot (x,y)]^2}}}.$$
\end{proof}

If the intersection of an extremal ray $\tau$ with the initial circle is a segment, the above result is also fulfilled by all points of $\FraS \cap \tau$. When the above mentioned intersection is only one point, the inequality we get is the inequality that appears in the proof of Proposition  \ref{ecuacion}.

\begin{example}
Consider the circle $C$ with center $(7/3,4/3)$ and radius $1/3.$ We are going to apply the algorithm shown in Theorem \ref{teorema_circulos_fg} to the semigroup $\FraS$ generated by $C.$

We compute the integer cone $\FraC$ delimited by the extremal rays  of  $L_{\Q_{\geq}}(C)$. This cone is minimally generated by $$\Big\{(4,3),(12,5),(2,1),(3,2),(7,3)\Big\}.$$
With the notation of Theorem \ref{teorema_circulos_fg}, $g_1=(4,3)$, $g_2=(12,5),$ $s_1=(32,24)=8g_1$ and $s_2=(96,40)=8g_2.$

Applying the construction of the system of generators of $\FraS'$ of (\ref{gen_rayos1}), the semigroup $\FraS'$ is minimally generated by
$$\Big\{(2,1),(3,2),(7,3),(7,5),(11,8),(15,11),(19,14),(23,17),(27,20),(31,23),$$
$$(32,24),(96,40),(19,8),(31,13),(43,18),(55,23),(67,28),(79,33),(91,38)\Big\}.$$
This semigroup is equal to  $\FraS$ in their extreme rays and equal to $\FraC$ in their interiors.

The finite set $\FraS'\setminus \FraS$ has 13 points.
By using Corollary \ref{semigrupo menos un numero finito de puntos},
we eliminate recurrently from $\FraS'$ the points of $\FraS'\setminus \FraS$ obtaining the minimal system of generators of  $\FraS$ (see Figure \ref{7_3-4_3-1_3_generators}): $$\Big\{(5,3),(6,4),(7,3),(7,4),(7,5),(8,4),(9,5),(9,6),(10,5),(11,6),(11,8),(13,6),$$
$$(15,11),(18,8),(19,14),(23,10),(23,17),(27,20),(31,23),(32,24),(33,14),(35,26),$$
$$(38,16),(50,21),(55,23),(67,28),(79,33),(91,38),(96,40),(115,48),(127,53),(139,58)\Big\}.$$
        \begin{figure}[h]
            \begin{center}
\includegraphics[scale=0.25]{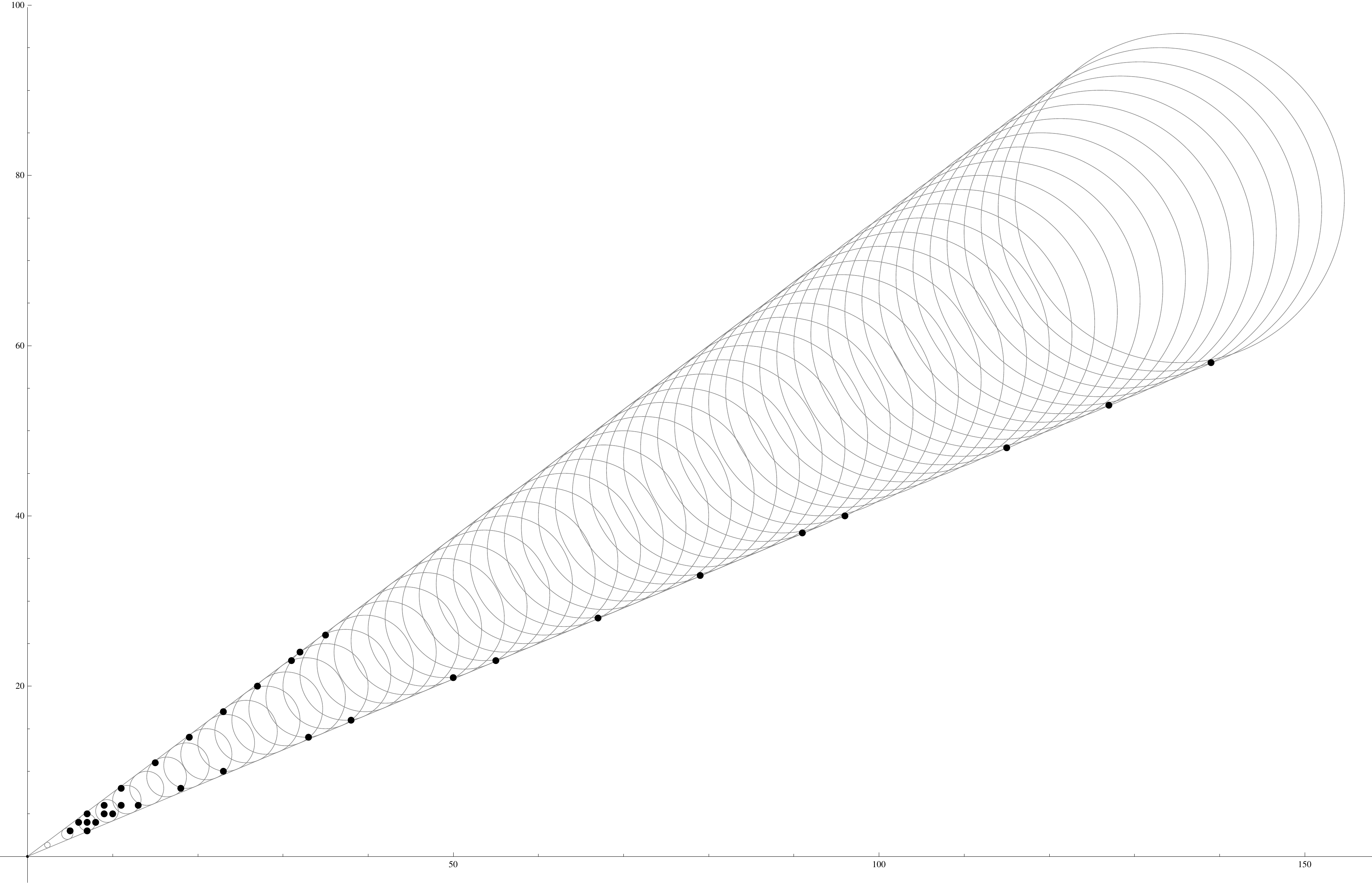}
            \caption{The minimal generators set of the semigroup generated by the circle with center $(7/3,4/3)$ and radius $1/3.$}\label{7_3-4_3-1_3_generators}
            \end{center}
        \end{figure}

This example has been computed by using our program {\tt CircleSG} available in \cite{programa} (this programm requires {\tt Wolfram Mathematica 7} to run).
\end{example}

\section{Bounding the minimal system of generators}\label{s6}
Assume that $\FraS$ is an affine semigroup obtained from a circle and consider the norm $||(x_1,\ldots, x_n)||_1=\sum _{i=1}^n |x_i|.$ Denote by $M$ the maximum of the norms of the elements of the minimal system of generators of the cone $\FraC$. One can find several bounds for this value (see \cite{Pottier91} and \cite{Sturmfels91}).

Following the notation given in the proof of Theorem \ref{teorema_circulos_fg}, denote by $l$ the cardinality of the finite set $\int(\FraS')\setminus \int(\FraS).$ Furthermore, the minimal elements of $\FraS$ in $\tau_1$ and $\tau_2$ are integer multiples of $g_1$ or $g_2$. Denote by $k$ the maximum of such integers.

\begin{corollary}\label{bound}
Every element $s$ of the minimal system of generators of $\FraS$ fulfills that $$|| s ||_1 \le 3^l (2k-1) M .$$
\end{corollary}

\begin{proof}
The minimal system of generators of  $\FraS'$ can be obtained from (\ref{gen_rayos1}), (\ref{gen_rayos2}) or (\ref{gen_rayos3}). Thus, the norm of their elements can bounded by the value
$$(2k-1)M=\max\{kM,M,(k-1)M +M ,kM+(k-1)M,(k-1)M+(k-1)M+M\},$$ where every value $\{kM,M,(k-1)M +M ,kM+(k-1)M,(k-1)M+(k-1)M+M\}$ is a bound for the elements of the subsets obtained in  (\ref{gen_rayos1}), (\ref{gen_rayos2}) and (\ref{gen_rayos3}).

To obtain a system of generators of $\FraS,$ we apply sequentially to the elements of  $\int(\FraS')\setminus \int(\FraS)$ the algorithm described in Corollary \ref{semigrupo menos un numero finito de puntos}. For the first iteration one has the bound is the maximum of $\{(2k-1)M,2(2k-1)M,3(2k-1)M\}$.

Since the above method is applied as many times as elements has the set $\int(\FraS')\setminus \int(\FraS)$, a bound for the elements of the minimal system of generators of $\FraS$ is $3^l(2k-1)M.$
\end{proof}

\begin{remark}
Analogously, a bound for the minimal generators of a convex polygonal semigroup can be obtained.
\end{remark}

\end{document}